\documentclass[a4paper,12pt]{amsart}
\usepackage{latexsym}
\usepackage{amsfonts}%\mathbb
\usepackage{amsmath}
\usepackage{mathrsfs}
\usepackage{amsthm}
\usepackage{amssymb}

\newcommand\E{{\mathbb E}}

\newtheorem{theorem}{Theorem}

\newtheorem{lemma}{Lemma}

\newcommand{\be}{\begin{eqnarray}}
\newcommand{\ee}{\end{eqnarray}}
\newcommand{\bestar}{\begin{eqnarray*}}
\newcommand{\eestar}{\end{eqnarray*}}
\newcommand{\ignore}[1]{}{}
\allowdisplaybreaks

\begin{document}

\vspace{1in}

\title[LIMIT THEOREMS FOR GENERALIZED RANDOM GRAPHS]{\bf LIMIT THEOREMS FOR NUMBER OF EDGES\\ IN THE GENERALIZED RANDOM GRAPHS\\ WITH RANDOM VERTEX WEIGHTS}

%\vspace{2in}

\author[Z.S. Hu]{Z.S. Hu}
\address{Z.S. Hu\\
	Department of Statistics and Finance\\
University of Science and Technology of China\\ 
Hefei, China
}
\email{huzs@ustc.edu.cn}

\author[V.V. Ulyanov]{V.V. Ulyanov}
\address{V.V. Ulyanov\\
	Faculty of Computational Mathematics and Cybernetics\\
	Moscow State University \\
	Moscow, 119991, Russia\\
	and National Research University Higher School of Economics (HSE),
	Moscow, 101000, Russia
}
\email{vulyanov@cs.msu.su}

\author[Q.Q. Feng]{Q.Q. Feng}
\address{Q.Q. Feng\\
		Department of Statistics and Finance\\
		University of Science and Technology of China\\ 
		Hefei, China
}
\email{fengqq@ustc.edu.cn}
\thanks{}


\begin{abstract}
We get central limit type theorems for the total number of edges in the generalized random graphs with random vertex weights under different moment conditions on the distributions of the weights.
\end{abstract}

\maketitle

\renewcommand{\refname}{References}

%\markboth{Section 1. Limit Theorems}{\emph{Ulyanov~V.V.},
%Cornish--Fisher Expansions}

 Complex networks attract increasing attention of researchers in various fields of science. In last years  numerous network models have been proposed. Since the uncertainty and the lack of regularity in real-world
 networks, these models are usually random graphs. Random graphs were first defined by Paul Erd\H{o}s and Alfr\'{e}d R\'{e}nyi in their 1959 paper "On Random Graphs", see \cite{UVV:ER}, and independently by Gilbert in \cite{UVV:G}. The suggested models are closely related: there are $n$ isolated vertices and  every possible edge occurs independently with probability $p:\, 0 < p < 1$. It is assumed that there are no self-loops. Later the models were generalized. 
 A natural generalization 
 of the Erd\H{o}s and R\'{e}nyi random graph is that the equal edge probabilities are replaced by probabilities depending on the vertex weights. Vertices with higher weights are more likely to have more  neighbors
 than vertices with small weights. Vertices with extremely high weights could act as the hubs observed in many real-world networks.
  
 The following generalized random graph model was first introduced by Britton et al., see \cite{UVV:Britt}. Let $\{1, 2, . . . , n\}$ be the set
 of vertices, and $W_i > 0$ be the weight of vertex $i,  1\leq i\leq n$. The edge probability of the edge between any two vertices $i$ and $j$ is equal to 
$$
p_{ij} = \frac
{W_i W_j}
{L_n + W_i W_j}
,
$$
where $L_n = \sum^n_{i=1} W_i$  denotes the total weight of all vertices, and the weights $W_i, i = 1, 2, \dots , n$ can be taken to be
deterministic or random. If we take all $W_i$-s as the same constant: 
$W_i \equiv n \lambda/(n - \lambda)$ for some $0 < \lambda < n$, it is easy to see that $p_{ij} = \lambda/n$ holds for
all $1 \leq i < j \leq n$. That is, the Erd\H{o}s--R\'{e}nyi random graph with $p = \lambda/n$ is a special case of the generalized random graph. 
There are many versions of the generalized random graphs, such as Poissonian random graph (introduced by Norros and Reittu in \cite{UVV:Norros}  and
studied by Bhamidi et al.\cite{UVV:Bhamidi}), rank-1 inhomogeneous random graph (see \cite{UVV:Bollob}), random graph with given
prescribed degrees (see \cite{UVV:Chung}), etc. Under some common conditions (see \cite{UVV:Janson}), all of the above mentioned random
graph models are asymptotically equivalent, meaning that all events have asymptotically equal probabilities. The updated review on the 
  results about these inhomogeneous random graphs see in Chapters 6 and 9 in \cite{UVV:Hofstad}.

In the present paper we assume that 
$W_i, i = 1, 2, \dots , n,$ 
are independent identically distributed random variables distributed as $W$. Let $E_n$ be the total number of edges in a generalized random graph with vertex weights $W_1, W_2, \dots , W_n.$ In   \cite{UVV:CFE}, under the conditions that $W$ has a finite or infinite mean, several weak
laws of large numbers for $E_n$ are established, see also Ch.6, \cite{UVV:Hofstad}. For instance, in \cite{UVV:CFE} and Ch.6, \cite{UVV:Hofstad}, it is proved that  $E_n/n$ tends in probability to $\E W/2$, provided $\E W$ is finite. 

Note that 
$$
E_n = \frac{1}{2}\,\sum^{n}_{i=1}D_i,
$$
where $D_i, i = 1, 2, \dots , n$ is a degree of vertex $i$, i.e. the number of edges coming out from  vertex $i$. It is clear, the random variables $D_i, i = 1, 2, \dots , n$ are dependent ones.  
 The aim of the present paper is to refine the law of large numbers type results for $E_n$ and to get central limit type theorems  under different moment conditions for $W$. In Theorem~1 we assume that $\E W^2 < \infty$. It implies normal limit distribution for $\{E_n\}$ after proper normalization. In Theorem~2 we assume that the distribution of $W$ belongs to the domain of attraction of a stable law $F$ with characteristic exponent $\alpha: 1 < \alpha < 2$. Then we prove that the limit distribution for normalized $E_n$ is $F$.

\begin{theorem} \label{thm1}
If $\E W^2<\infty$, then
\bestar
\frac{2E_n-n\E W}{\sqrt{n\,(2\E W+\mbox{Var}(W))}}\stackrel{d}{\longrightarrow} N(0,1).
\eestar
\end{theorem}

\begin{proof}
Put for all integer $n\geq 1$
%$b_n=(1/2)nEW,~c_n=(1/2)\sqrt{n\mbox{Var}(W)}$.
\be\label{hu0}
b_n=\frac{1}{2}n\,\E W,~c_n=\frac{1}{2}\sqrt{n\,\mbox{Var}(W)}.
\ee
For any $t\in \mathbb{R}$, we have
\be
\E\exp\Big\{{it}\frac{E_n-b_n}{c_n}\Big\}&=& \E\exp\Big\{\frac{it}{c_n}\Big(\sum\limits_{1\le i<j\le n}I_{ij}-b_n\Big)\Big\}\nonumber\\
&=& \E\Big(\E\Big(\exp\Big\{\frac{it}{c_n}\Big(\sum\limits_{1\le i<j\le n}I_{ij}-b_n\Big)\Big\}\Big| W_1, \cdots, W_n\Big)\Big)\nonumber\\
&=& \E\Big(e^{-itb_n/c_n}\prod_{1\le i<j\le n}\frac{L_n+e^{it/c_n}W_iW_j}{L_n+W_iW_j}\Big)\nonumber\\
&:=& \E e^{Y_n}, \label{hu1}\nonumber
\ee
where
\be
Y_n&=&\sum_{1\le i<j\le n}\log\frac{L_n+e^{it/c_n}W_iW_j}{L_n+W_iW_j}-\frac{itb_n}{c_n}\nonumber\\
&=&\frac{1}{2} \sum_{i=1}^n\sum_{j=1}^ n\log\frac{L_n+e^{it/c_n}W_iW_j}{L_n+W_iW_j}-\frac{itb_n}{c_n}
-\sum_{i=1}^n \log\frac{L_n+e^{it/c_n}W_i^2}{L_n+W_i^2} \label{hu2}
\ee
and $\log(\cdot)$ is  the principal value of the complex logarithm function.

By using the Maclaurin series expansion  of $\log (1+x)$ for complex $x$ with $|x|<1$, we have  that
\bestar
\frac{|\log(1+x)|}{|x|}\longrightarrow 1,~~\frac{|\log(1+x)-x|}{|x|^2}\longrightarrow \frac{1}{2}~~~~\mbox{as}~~~~|x|\rightarrow 0.
\eestar
Hence
there exists some constant $c_0>0$ such that $|\log(1+x)|\le 2|x|$ and $|\log(1+x)-x|\le |x|^2$
hold for any $|x|\le c_0$.

Clearly, for any fixed $t$, there exists $n_0=n_0(t) \in \mathbb{N}$
such that for all $n\ge n_0$ and any $1\leq i, j \leq n$ one has
\bestar
\Big|\frac{(e^{it/c_n}-1)W_iW_j}{L_n+W_iW_j}\Big|\le |e^{it/c_n}-1|\le |t|/c_n \le c_0.
%~~~~n\ge n_0.
\eestar
Thus, since
\be\label{hu00}
\frac{ L_n}{n}  \rightarrow \E W~~a.s.~~ \mbox{and}~~ \frac{\sum_{i=1}^n W^2_i}{n}  \rightarrow \E W^2 ~~a.s., %\nonumber
\ee
we have  for any $n\ge n_0$
\be
\Big|\sum_{i=1}^n \log\frac{L_n+e^{it/c_n}W_i^2}{L_n+W_i^2}\Big|&\le& \sum_{i=1}^n \Big|\log\Big(1+\frac{(e^{it/c_n}-1)W_i^2}{L_n+W_i^2}\Big)\Big|\nonumber\\
&\le& 2|e^{it/c_n}-1|\sum_{i=1}^n\frac{W_i^2}{L_n+W_i^2}\nonumber\\
&\le& 2\frac{|t|}{c_n}\frac{\sum_{i=1}^n W_i^2}{L_n}\rightarrow 0~~a.s. \label{hu3}
\ee
and
\be
&&~~~~\frac{1}{2} \sum_{i=1}^n\sum_{j=1}^ n\log\frac{L_n+e^{it/c_n}W_iW_j}{L_n+W_iW_j}-\frac{itb_n}{c_n}\nonumber\\
&&=\frac{1}{2} \sum_{i=1}^n\sum_{j=1}^ n\log\Big(1+\frac{(e^{it/c_n}-1)W_iW_j}{L_n+W_iW_j}\Big)-\frac{itb_n}{c_n}\nonumber\\
&&=\frac{1}{2} \sum_{i=1}^n\sum_{j=1}^ n\frac{(e^{it/c_n}-1)W_iW_j}{L_n+W_iW_j}-\frac{itb_n}{c_n}
+O_1 \sum_{i=1}^n\sum_{j=1}^ n\frac{(e^{it/c_n}-1)^2W_i^2W_j^2}{(L_n+W_iW_j)^2}\nonumber\\
&&=\frac{1}{2}\Big(e^{it/c_n}-1-\frac{it}{c_n}+\frac{t^2}{2c_n^2}\Big) \sum_{i=1}^n\sum_{j=1}^ n\frac{W_iW_j}{L_n+W_iW_j}\nonumber\\
&&~~~~~~+\frac{1}{2}\Big(\frac{it}{c_n}-\frac{t^2}{2c_n^2}\Big) \sum_{i=1}^n\sum_{j=1}^ n\frac{W_iW_j}{L_n+W_iW_j}-\frac{itb_n}{c_n}\nonumber\\
&&~~~~~~
+O_1 \sum_{i=1}^n\sum_{j=1}^ n\frac{(e^{it/c_n}-1)^2W_i^2W_j^2}{(L_n+W_iW_j)^2}\nonumber\\
&&:=I_1+I_2+I_3, \label{hu3-1}
\ee
where $|O_1|\le 1/2$.
By (\ref{hu00}) and the inequality $|e^{ix}-1-ix+x^2/2|\le  |x|^3/6$ for any $x\in \mathbb{R}$, we have
\be
|I_1| \le \frac{|t|^3}{12c_n^3} \sum_{i=1}^n\sum_{j=1}^ n\frac{W_iW_j}{L_n}
=\frac{|t|^3L_n}{12c_n^3}{\longrightarrow} 0~~a.s. \label{hu4}
\ee
Similarly, by  (\ref{hu00}) and the inequality $|e^x-1|\le |x|$, we get
\be
|I_3|\le  \frac{t^2}{2c_n^2}\sum_{i=1}^n\sum_{j=1}^ n\frac{W_i^2W_j^2}{L_n^2}=\frac{t^2}{c_n^2}\Big(\frac{1}{n}\sum_{i=1}^n W_i^2\Big)^2\Big(\frac{n}{L_n}\Big)^2
 {\longrightarrow} 0~~a.s.
\ee
Recalling the definition (\ref{hu0}) for $b_n$ and $c_n$,
%that $b_n=(1/2)nEW,~c_n=(1/2)\sqrt{n\mbox{Var}(W)}$,
we have
\bestar
I_2&=&\frac{1}{2}\Big(\frac{it}{c_n}-\frac{t^2}{2c_n^2}\Big)\Big( \sum_{i=1}^n\sum_{j=1}^ n\frac{W_iW_j}{L_n}- \sum_{i=1}^n\sum_{j=1}^ n\frac{W_i^2W_j^2}{L_n(L_n+W_iW_j)}\Big)-\frac{itb_n}{c_n}\\
&=& it\frac{L_n-n\E W}{\sqrt{n\mbox{Var}(W)}}-\frac{t^2L_n}{n\mbox{Var}(W)}-
\frac{1}{2}\Big(\frac{it}{c_n}-\frac{t^2}{2c_n^2}\Big) \sum_{i=1}^n\sum_{j=1}^ n\frac{W_i^2W_j^2}{L_n(L_n+W_iW_j)}.
\eestar
Moreover, by (\ref{hu00}) we get
%Since $L_n/n\rightarrow EW~a.s.$ and
\bestar
\sum_{i=1}^n\sum_{j=1}^ n\frac{W_i^2W_j^2}{L_n(L_n+W_iW_j)}&&\le \sum_{i=1}^n\sum_{j=1}^ n\frac{W_i^2W_j^2}{L_n^2}
\\ &&=\frac{\Big(\sum_{i=1}^nW_i^2\Big)^2}{L_n^2}\rightarrow \Big(\frac{\E W^2}{\E W}\Big)^2~~a.s.
\eestar
 The central limit theorem yields
 \be
 I_2\stackrel{d}{\longrightarrow} it{\bf N}- t^2 \E W/\mbox{Var}(W),
 \label{hu5}
 \ee
 where ${\bf N}$ is a standard normal random variable.
 Now, it follows from (\ref{hu2})--(\ref{hu5}) that
 \bestar
 Y_n \stackrel{d}{\longrightarrow} it{\bf N}- t^2 \E W/\mbox{Var}(W).
 \eestar

Hence, by noting that $|e^{Y_n}|\le 1$ and  applying the Lebesgue dominated convergence theorem, we get that, for any $t\in \mathbb{R}$,
 \bestar
\E\exp\Big\{{it}\frac{E_n-b_n}{c_n}\Big\}&=&\E e^{Y_n}\rightarrow \E\exp\{it{\bf N}- t^2 \E W/\mbox{Var}(W)\}\\
 &=&\exp\{-(1/2)t^2(1+2\E W/\mbox{Var}(W))\}.
 \eestar
 Thus, Theorem \ref{thm1} is proved.
\end{proof}

In the following theorem we get convergence of the sequence $\{E_n\}$ under weaker moment conditions on $W_i$'s.

\begin{theorem} \label{thm2}
Let $W, W_1,W_2,\cdots$ be a sequence of  i.i.d. nonnegative random variables and
\be
\frac{W_1+\cdots+W_n-n\E W}{a_n}\stackrel{d}{\longrightarrow} F,
\label{hu200}
\ee
where $F$ is a stable distribution with characteristic exponent $\alpha : 1<\alpha<2$,
then
\bestar
\frac{2E_n-n\E W}{a_n}\stackrel{d}{\longrightarrow} F.
\eestar
\end{theorem}

Before we start to prove the theorem, let us state some properties of the distribution of $W$.
%Theorem ??? in
%Feller(1971),  I

If (\ref{hu200}) holds true, then $a_n$ (see e.g. \cite{UVV:FUS}, ch.XVII, \S 5) is a regularly varying function with exponent $1/\alpha$ satisfying
\be
n\E W^2I(W\le a_n)\sim a_n^2, \label{hu204}%\nonumber
\ee
 and there exists some constant $c>0$ and $h(x)$, a slowly varying function at $\infty$,
such that
\be
P(W>x)\sim cx^{-\alpha}h(x). \label{hu205}
\ee
We shall use the following lemma.

\begin{lemma} \label{lemma1}
If (\ref{hu205}) holds with $\alpha: 1<\alpha<2$, then we have
\bestar
&&~~\E W^2I(W\le x) \sim \frac{c\alpha}{2-\alpha} x^{2-\alpha}h(x),\\
&&~~\E WI(W\ge x) \sim c\,\frac{2-\alpha}{\alpha-1} x^{1-\alpha}h(x).
\eestar
\end{lemma}

The proof of the lemma see e.g. \cite{UVV:FUS}, ch.XVII, \S 5.

Now we are ready to prove Theorem 2.

\begin{proof} Let $b_n=(1/2)\,n\,\E W$ and $c_n=(1/2)\,a_n$ with $a_n$ from (\ref{hu204}).
As in the proof of Theorem \ref{thm1}, for any $t\in \mathbb{R}$, we also write
\bestar
\E\exp\Big\{{it}\frac{E_n-b_n}{c_n}\Big\}
= \E e^{Y_n}
\eestar
with new definition for $c_n$ and
\bestar
Y_n=\frac{1}{2} \sum_{i=1}^n\sum_{j=1}^ n\log\frac{L_n+e^{it/c_n}W_iW_j}{L_n+W_iW_j}-\frac{itb_n}{c_n}
-\sum_{i=1}^n \log\frac{L_n+e^{it/c_n}W_i^2}{L_n+W_i^2}.
\eestar
For the last sum for any $n\ge n_0$, where $n_0=n_0(t)$ is defined in the proof of Theorem \ref{thm1}, we have (cp. (\ref{hu3}))
\bestar
\Big|\sum_{i=1}^n \log\frac{L_n+e^{it/c_n}W_i^2}{L_n+W_i^2}\Big|
\le 2\frac{|t|}{c_n}\frac{\sum_{i=1}^n W_i^2}{L_n}.
\eestar
Similarly to (\ref{hu3-1}), we get
\bestar
&&~~~~\frac{1}{2} \sum_{i=1}^n\sum_{j=1}^ n\log\frac{L_n+e^{it/c_n}W_iW_j}{L_n+W_iW_j}-\frac{itb_n}{c_n}\nonumber\\
&&=\frac{1}{2} \sum_{i=1}^n\sum_{j=1}^ n\frac{(e^{it/c_n}-1)W_iW_j}{L_n+W_iW_j}-\frac{itb_n}{c_n}
+O_1 \sum_{i=1}^n\sum_{j=1}^ n\frac{(e^{it/c_n}-1)^2W_i^2W_j^2}{(L_n+W_iW_j)^2}\nonumber\\
&&=\frac{1}{2}\Big(e^{it/c_n}-1-\frac{it}{c_n}\Big) \sum_{i=1}^n\sum_{j=1}^ n\frac{W_iW_j}{L_n+W_iW_j}\nonumber\\
&&~~~~~~+\frac{1}{2}\frac{it(L_n-2b_n)}{c_n}-
\frac{1}{2}\frac{it}{c_n} \sum_{i=1}^n\sum_{j=1}^ n\frac{W_i^2W_j^2}{L_n (L_n+W_iW_j)}\nonumber\\
&&~~~~~~
+O_1 \sum_{i=1}^n\sum_{j=1}^ n\frac{(e^{it/c_n}-1)^2W_i^2W_j^2}{(L_n+W_iW_j)^2}
\eestar
with $|O_1|\le 1/2$. Due to Theorem's condition we have $(L_n-2b_n)/(2c_n)\stackrel{d}{\rightarrow} F$. Since
\bestar
|e^{ix}-1|\le |x|,~~ |e^{ix}-1-ix|\le  |x|^2/2 ~~~~\mbox{for all}~~~~ x\in \mathbb{R},
\eestar
 in order to
prove Theorem \ref{thm2}, we only need to show that
 \be
 &&\frac{1}{a_n}\frac{\sum_{i=1}^n W_i^2}{L_n}\stackrel{p}{\longrightarrow} 0, \label{hu20-1}\\
 &&\frac{1}{a_n^2}\sum_{i=1}^n\sum_{j=1}^ n\frac{W_iW_j}{L_n+W_iW_j}\stackrel{p}{\longrightarrow} 0, \label{hu201}\\
 && \frac{1}{a_n}\sum_{i=1}^n\sum_{j=1}^ n\frac{W_i^2W_j^2}{L_n (L_n+W_iW_j)}\stackrel{p}{\longrightarrow} 0, \label{hu202}\\
 &&\frac{1}{a_n^2}\sum_{i=1}^n\sum_{j=1}^ n\frac{W_i^2W_j^2}{(L_n+W_iW_j)^2}
 \stackrel{p}{\longrightarrow} 0. \label{hu203}
 \ee
For any $\gamma: \alpha > \gamma>0$, we have $\E (W^2)^{(\alpha-\gamma)/2}=\E W^{\alpha-\gamma}<\infty$. Then
by Marcinkiewicz--Zygmund's strong law of large numbers (see e.g. Theorem 4.23 in \cite{UVV:KAL})
we have
\bestar
n^{-2/(\alpha-\gamma)}\sum_{i=1}^n W_i^2\rightarrow 0~~a.s.
\eestar
Since $a_n$ is a regularly varying function with exponent $1/\alpha$,
 then we have $1/a_n=o(n^{-1/\alpha+\gamma})$.
Now choose $\gamma>0$ such that
$$2/(\alpha-\gamma)-1-1/\alpha+\gamma<0~~~~\mbox{ and}~~~~ -2/\alpha+1+2\gamma<0.$$
Then we have
\bestar
\frac{1}{a_n}\frac{\sum_{i=1}^n W_i^2}{L_n}=\frac{n^{2/(\alpha-\gamma)-1}}{a_n}\frac{\sum_{i=1}^n W_i^2/n^{2/(\alpha-\gamma)}}{L_n/n}=o(n^{2/(\alpha-\gamma)-1-1/\alpha+\gamma})\longrightarrow 0~~a.s.
\eestar
and
\bestar
\frac{1}{a_n^2}\sum_{i=1}^n\sum_{j=1}^ n\frac{W_iW_j}{L_n+W_iW_j}\le
\frac{1}{a_n^2}\sum_{i=1}^n\sum_{j=1}^ n\frac{W_iW_j}{L_n}=\frac{n}{a_n^2}\frac{L_n}{n} =o(n^{-2/\alpha+1+2\gamma})
{\longrightarrow} 0~~a.s.
\eestar
Thus we get (\ref{hu20-1}) and (\ref{hu201}).

To prove (\ref{hu202}), we write
\bestar
&& ~~~~\frac{1}{a_n}\sum_{i=1}^n\sum_{j=1}^ n\frac{W_i^2W_j^2}{L_n (L_n+W_iW_j)}\\
&&= \frac{1}{a_n}\sum_{i=1}^n\sum_{j=1}^ n\frac{W_i^2W_j^2I(W_iW_j\le n)}{L_n (L_n+W_iW_j)}
+\frac{1}{a_n}\sum_{i=1}^n\sum_{j=1}^ n\frac{W_i^2W_j^2I(W_iW_j>n)}{L_n (L_n+W_iW_j)}\\
&&\le \frac{1}{a_n}\sum_{i=1}^n\sum_{j=1}^ n\frac{W_i^2W_j^2I(W_iW_j\le n)}{L_n^2}
+\frac{1}{a_n}\sum_{i=1}^n\sum_{j=1}^ n\frac{W_iW_jI(W_iW_j>n)}{L_n }\\
&&\le \frac{n^2}{a_nL_n^2}\sum_{i=1}^n\sum_{j=1}^ n\frac{W_i^2W_j^2I(W_iW_j\le n)}{n^2}
+\frac{n}{a_nL_n}\sum_{i=1}^n\sum_{j=1}^ n\frac{W_iW_jI(W_iW_j>n)}{n}.
\eestar
Further, by (\ref{hu00})
%by noting that  $L_n/n\rightarrow EW$
and by using the fact that $\E|X_n| \rightarrow 0$ implies $X_n\stackrel{p}{\rightarrow} 0$,
in order to prove (\ref{hu202}), it is sufficient to show that
\be
&&\frac{1}{a_n}\E W_1^2W_2^2I(W_1W_2\le n){\longrightarrow} 0,\label{hu208}\\
&&\frac{n}{a_n}\E W_1W_2I(W_1W_2> n){\longrightarrow} 0. \label{hu209}
\ee
For any $\alpha\in (1,2)$, we can choose $\delta>0$ satisfying $2-\alpha-1/\alpha+2\delta<0$.
By Lemma \ref{lemma1},  there exists some constant $c_1=c_1(\alpha,\delta)>0$ such that
\bestar
\E W^2I(W\le x) \le c_1 x^{2-\alpha+\delta},~~
~\E WI(W\ge x)\le c_1 x^{1-\alpha+\delta}
\eestar
hold for all $x>1$. Hence
\bestar
&&~~\frac{1}{a_n}\E W_1^2W_2^2I(W_1W_2\le n)\\
&&=\frac{1}{a_n}\E\Big(W_2^2I(W_2\le n)\E (W_1^2I(W_1\le n/W_2)|W_2)\Big)\\
&&~~~~~~
+\frac{1}{a_n}\E\Big(W_2^2I(W_2> n)\E (W_1^2I(W_1\le n/W_2)|W_2)\Big)\\
&&\le \frac{c_1}{a_n}\E\Big(W_2^2( n/W_2)^{2-\alpha+\delta}\Big)
+\frac{1}{a_n}\E\Big(W_2^2I(W_2> n)(n/W_2)^2\Big)\\
&&=\frac{c_1n^{2-\alpha+\delta}}{a_n}\E W^{\alpha-\delta}
+\frac{n^2}{a_n}P(W> n).
\eestar
Since by (\ref{hu205}) we have
$
P(W>n)\sim cn^{-\alpha} h(n)=o(n^{-\alpha+\delta})
$ and $1/a_n=o(n^{-1/\alpha+\delta})$, we get
\bestar
\frac{1}{a_n}\E W_1^2W_2^2I(W_1W_2\le n)=o(n^{2-\alpha-1/\alpha+2\delta})\rightarrow 0~~~\mbox{as}~~~x\rightarrow \infty.
\eestar
Thus, we get (\ref{hu208}).

Similarly, we have
\bestar
&&~~\frac{n}{a_n}\E W_1W_2I(W_1W_2> n)\\
&&=\frac{n}{a_n}\E\Big(W_2I(W_2\le n)\E(W_1I(W_1> n/W_2)|W_2)\Big)\\
&&~~~~~~
+\frac{n}{a_n}\E\Big(W_2I(W_2> n)\E (W_1I(W_1> n/W_2)|W_2)\Big)\\
&&\le \frac{c_1n}{a_n}\E\Big(W_2(n/W_2)^{1-\alpha+\delta}\Big)
+\frac{n}{a_n}\E\Big(W_2I(W_2> n)\E W_1\Big)\\
&&=\frac{c_1n^{2-\alpha+\delta}}{a_n}\E W^{\alpha-\delta}
+\frac{n}{a_n}\E W \E(WI(W> n))\\
&&\le \frac{c_1n^{2-\alpha+\delta}}{a_n}\E W^{\alpha-\delta}
+\frac{c_1n^{2-\alpha+\delta}}{a_n}\E W=o(n^{2-\alpha-1/\alpha+2\delta})\rightarrow 0.
\eestar
Hence (\ref{hu209}), and then (\ref{hu202}), are proved.

 And (\ref{hu203}) follows from (\ref{hu202}).
The proof of Theorem \ref{thm2} is complete.
\end{proof}

\end{document}